\documentclass [11pt]{amsart}
\usepackage{amsmath}
\usepackage{amssymb}
\usepackage{comment}

\newtheorem{thm}{Theorem}[section]

\newtheorem{lem}[thm]{Lemma}

\theoremstyle{definition}

\theoremstyle{remark}
\newtheorem{rem}[thm]{Remark}
\numberwithin{equation}{section}

\newcommand{\beas}{\begin{eqnarray*}}
\newcommand{\eeas}{\end{eqnarray*}}
\newcommand{\bes} {\begin{equation*}}
\newcommand{\ees} {\end{equation*}}
\newcommand{\be} {\begin{equation}}
\newcommand{\ee} {\end{equation}}
\newcommand{\bea} {\begin{eqnarray}}
\newcommand{\eea} {\end{eqnarray}}
\newcommand{\ra} {\rightarrow}
\newcommand{\txt} {\textmd}
\newcommand{\ds} {\displaystyle}
\newcommand{\R} {\mathbb{R}}
\newcommand{\C} {\mathbb{C}}
\begin{document}

\title[An Uncertainty Principle of Paley and Wiener on Euclidean Motion Group] {An Uncertainty Principle of Paley and Wiener on Euclidean Motion Group}

\author{Mithun Bhowmik and Suparna Sen}

\address{Stat-Math Unit, Indian Statistical Institute, 203 B. T. Road, Kolkata - 700108, India.}

\email{mithunbhowmik123@gmail.com, suparna29@gmail.com}

\thanks{The second author was supported by INSPIRE Faculty Award from Department of Science and Technology, India.}


\begin{abstract}
A classical result due to Paley and Wiener characterizes the existence of a non-zero function in $L^2(\R)$, supported on a half line, in terms of the decay of its Fourier transform. In this paper we prove an analogue of this result for compactly supported continuous functions on the Euclidean motion group $M(n)$. We also relate this result to a uniqueness property of solutions to the initial value problem for time-dependent Schr\"odinger equation on $M(n)$.  
\vspace{0.1in}
\begin{flushleft}
MSC 2010 : Primary 22E30; Secondary 43A80. \\
Keywords : Uncertainty Principle, Euclidean Motion Group, Schr\"odinger Equation. \\
\end{flushleft}
\end{abstract}

\maketitle

\section{Introduction}

An uncertainty principle in harmonic analysis says that a non-zero function and its Fourier transform both cannot be small simultaneously. Depending on various notions of smallness one gets different versions of uncertainty principle on Euclidean spaces and various non-commutative groups (see \cite{FS, T, S}). In this paper we will be concerned with an uncertainty principle involving the support of a function and the pointwise decay of its Fourier transform. It is well known that if $f$ is compactly supported in $\R$ and its Fourier transform satisfies the estimate $|\hat f(y)| \leq Ce^{-a|y|}$ for some $a>0$, then $f$ extends as a holomorphic function to a strip in the complex plane containing the real line and hence $f$ is identically zero. However, the situation is not so simple if we consider a locally integrable function $\theta(y)$ instead of $a|y|$ and consider the following estimate on the Fourier transform
\bes |\hat f(y)| \leq Ce^{-\theta(|y|)}, \:\:\:\: \txt{ for } y \in \R.\ees A result due to Paley and Wiener (Theorem II, \cite{PW1}; Theorem XII, P. 16, \cite{PW}) gives a characterization of the existence of a non-zero function whose Fourier transform satisfies such an estimate in terms of an integrability condition on $\theta$. Our main concern in this paper is the result of Paley and Wiener.
\begin{thm} \label{paleywiener}
Let $\theta$ be a non-negative locally integrable function on $[0,\infty)$. There exists a non-zero $f \in L^2(\R)$ vanishing for $x \geq x_0$ for some $x_0 \in \R$ such that \bes|\hat{f}(y)| \leq C e^{-\theta(|y|)}, \:\:\:\: \txt{ for almost every } y \in \R, \ees if and only if \bes \int_0^{\infty}\frac{\theta (t)}{1+t^2}dt<\infty.\ees
\end{thm} 
\noindent Similar results have also been obtained in \cite{I, L, Hi, L1, K}.

The main purpose of this paper is to prove an analogue of Theorem \ref{paleywiener} on the Euclidean motion group $M(n)$ (see Theorem \ref{mn}). Since there is no obvious analogue of the half line on $M(n)$ we choose to concentrate on a weaker version of Theorem \ref{paleywiener}. Precisely, we will concentrate only on  compactly supported continuous functions defined on $M(n)$. The proof of Theorem \ref{mn} needs a several variable analogue of Theorem \ref{paleywiener}, for compactly supported continuous functions, which we prove in Theorem \ref{ccradon}. The main ingredient for the proof of Theorem \ref{mn} is a complex analytic lemma (see Lemma \ref{lemma}). It seems to us that this lemma can be applied to prove analogous results for other non-commutative groups and homogeneous spaces. In fact, we will use this lemma to prove an analogue of Theorem \ref{paleywiener} for connected, non-compact, semisimple Lie groups with finite center, in a forthcoming paper.

It has recently been observed in \cite{EKPV, KPV} that it is possible to relate some of the uncertainty principles to the problem of uniqueness of solutions to the Schr\"odinger Equation. Our last result (Theorem \ref{uniqG}) in this paper attempts to relate the uncertainty principle of Paley and Wiener to the uniqueness of solutions to the initial value problem for the time-dependent Schr\"odinger Equation on $M(n)$. 

The paper is organized as follows: In the next section we prove an analogue of Theorem \ref{paleywiener} for compactly supported continuous functions on $\R^n$ and a complex analytic lemma which is crucial in proving the main result. In section 3, we first describe the required preliminaries on the Euclidean motion group $M(n)$ and then prove the main theorem of this paper. In the last section we prove the uniqueness of solutions to the initial value problem for time-dependent Schr\"odinger equation on $\R^n$ and subsequently on $M(n)$. 

We will use the following notations and conventions in the paper: $\Im z$ denotes the imaginary part of $z$, $C_c(X)$ denotes the set of compactly supported continuous functions on $X$, $C_c^{\infty}(X)$ denotes the set of compactly supported smooth functions on $X$, supp$(f)$ denotes the support of the function $f$ and $C$ denotes a constant whose value may vary. For $x,y \in \R^n$, we will use $\|x\|$ to denote the norm of the vector $x$ and $x \cdot y$ to denote the Euclidean inner product of the vectors $x$ and $y$.

\section{Some Euclidean Results}
In this section we will first prove a weaker analogue of Theorem \ref{paleywiener} for $\R^n$. Next we will prove a complex analytic lemma which will then be used to prove an analogue of Theorem \ref{paleywiener} for $M(n)$. The main idea behind the proof of a several variable version of Theorem \ref{paleywiener} is to reduce matters to the one dimensional situation. This will be achieved by a simple application of the Radon transform. We will now briefly recall some standard facts regarding Radon transform which are important for us. We refer the reader to \cite{He} for proof of these results. 

For $\omega \in S^{n-1}$ and $t \in \R$ let $H_{\omega,t}=\{x\in\R^n:x \cdot \omega=t\}$ denote the hyperplane on $\R^n$ with normal $\omega$ and distance $|t|$ from the origin. For $f \in C_c(\R^n)$ the Radon transform $Rf$ of the function $f$ is defined by
\begin{equation}
Rf(\omega,t)=\int_{H_{\omega ,t}}f(x)dm(x),\nonumber
\end{equation} 
where $dm(x)$ is the $n-1$ dimensional Lebesgue measure of $H_{\omega,t}$. For $f \in L^1(\R^n)$ we define the Fourier transform $\hat{f}$ of $f$ by 
\begin{equation}
\hat{f} (y) = \int_{\R^n} f(x) e^{-ix \cdot y} dx,  \:\:\:\: \txt{ for } y \in \R^n.\nonumber
\end{equation} 
The one dimensional Fourier transform of $Rf$ and the Fourier transform of $f$ are closely connected by the slice projection theorem:
\begin{equation}
\label{sliceproj}
\hat{f}(\lambda\omega)= \mathcal F (Rf(\omega, \cdot))(\lambda),
\end{equation}
where $\mathcal F(Rf(\omega, \cdot))$ denotes the one dimensional Fourier transform of the function $t\mapsto Rf(\omega ,t)$.
Let $C_c^{\infty}(\R^n)_0$ denote the set of the compactly supported, smooth, radial functions on $\R^n$ and $C_c^{\infty}(\R)_e$ denote the set of compactly supported, smooth, even functions on $\R$. By Theorem 2.10, P. 14 of \cite{He} it is known that
\be
\label{radonmapping}
R:C_c^{\infty}(\R^n)_0\longrightarrow C_c^{\infty}(\R)_e
\ee
is a bijection. We are now in a position to state and prove an analogue of Theorem \ref{paleywiener} for $\R^n$.

\begin{thm}\label{ccradon}
Let $\theta:\R^n\rightarrow [0,\infty)$ be a radial, locally integrable function and
\begin{equation}
I=\int_{\|y\|\geq 1}\frac{\theta (y)}{\|y\|^{n+1}}dy.\label{integral}
\end{equation}
\begin{enumerate}
\item[a)] If $f\in C_c(\R^n)$ satisfies the estimate
\begin{equation}
|\hat{f}(y)|\leq Ce^{-\theta (y)},\label{estimate} \:\:\:\: \txt{ for } y \in \R^n,
\end{equation}
and $I=\infty$ then $f(x)=0$, for all $x\in\R^n$.
\item[b)] If the radial function $\theta$ on $\R^n$ is non-decreasing as a function on $[0,\infty)$ and $I$ is finite then there exists a non-zero radial $f\in C_c^{\infty}(\R^n)$ satisfying (\ref{estimate}).
\end{enumerate}
\end{thm}

\begin{proof}

To prove (a) we fix $\omega\in S^{n-1}$ and define $h_\omega(t) = Rf(\omega,t)$ for $t\in\R$. Clearly $h_\omega$ is compactly supported in $\R$. Since $\theta$ is radial, we can interpret it as an even function on $\R$ and from (\ref{sliceproj}) and (\ref{estimate}) we get that for $\lambda \in \R$, 
\begin{equation*}|\mathcal F(h_\omega)(\lambda)| = |\hat f(\lambda\omega)| \leq Ce^{-\theta(\lambda)}.
\end{equation*} 
As $\theta$ is a radial function it follows from the divergence of the integral $I$ that
\begin{equation*}
 \int_{0}^{\infty} {\frac{\theta(r)}{1+r^2}} dr = \infty.
\end{equation*} We now apply Theorem \ref{paleywiener} to conclude that  $h_\omega=0$. Since this holds for each $\omega\in S^{n-1}$, it follows from (\ref{sliceproj}) and uniqueness of the Fourier transform that $f=0$. This proves (a).

We now assume that $I$ is finite and construct a non-zero $f \in C_c^{\infty}(\R^n)$ as in (b). As before, using polar coordinates we now have that
\begin{equation*}
\int_{0}^{\infty} {\frac{\theta(r)}{1+r^2}} dr < \infty.
\end{equation*}
By Lemma 4 of \cite{L} there exists a non-zero $g_1 \in C_c(\R)$ such that 
\begin{equation} \label{g1decay}
|\mathcal F g_1(y)| \leq C e^{-\theta(y)},\:\:\:\:\text{for all $y \in \R$.}
\end{equation} 
Convolving with a $C_c^\infty(\R)$ function if necessary we may assume that $g_1 \in C_c^{\infty}(\R)$ and $g_1$ satisfies (\ref{g1decay}). We now consider the even function $g$ defined by  
\begin{equation*}
g(x) = \frac{g_1(x) + g_1(-x)}{2}, \:\:\:\:\txt{ for all $x \in \R$.}
\end{equation*}
Since the translates of $g_1$ also satisfy (\ref{g1decay}), we can ensure that $g$ is non-zero by translating $g_1$ suitably if necessary. 
It is clear that $g \in C_c^{\infty}(\R)$ is even with 
\begin{equation*}
|\mathcal F g(y)| \leq C e^{-\theta(y)},\:\:\:\:\txt{ for all $y \in \R$.}
\end{equation*} 
By (\ref{radonmapping}) we conclude that there exists a radial $f\in C_c^\infty(\R^n)$ such that $Rf=g$. It now follows from (\ref{sliceproj}) and the last inequality that the Fourier transform of $f$ satisfies the estimate
\begin{equation*}
| \hat{f}(y)|= |\mathcal F g(\|y\|)| \leq Ce^{-\theta(y)},\:\:\:\:\txt{for all $y\in\R^n$.}
\end{equation*}
This completes the proof of $b)$.
\end{proof}
\begin{rem}
The exact analogue of Theorem \ref{paleywiener} for $L^2$ functions supported on a half space in $\R^n$ is not hard to prove. This result and various other generalizations of Theorem \ref{paleywiener} will be discussed in a forthcoming paper of ours.
\end{rem}
Next, we will prove a complex analytic lemma on entire functions which will be used in the proof of the main theorem. The main idea of the proof of the following lemma can be traced back to \cite{Hi}.

\begin{lem}\label{lemma}
Let $f$ be an entire function on $\C$ and $\theta$ be a non-negative measurable even function on $\R$ such that for a positive constant $a$, we have 
\bea
|f(z)| &\leq& Ce^{a|z|}, \:\:\:\: \txt{ for all } z\in \C,\label{expdecay} \\ 
|f(x)| &\leq& Ce^{-\theta(x)}, \:\:\:\:  \txt{ for all } x\in \R. \label{decay}
\eea
If $\ds{\int_\R{\frac{\theta(t)}{1+t^2}dt}=\infty}$ then $f=0$ on $\C$.
\end{lem}

To prove the lemma, we will need two results. The first one is a standard result regarding upper semicontinuous functions. 
\begin{thm} [Theorem 3.6, P. 218, \cite{C}] \label{uppsemicont}
 Let $(X,d)$ be a metric space, $v: X\rightarrow [-\infty, \infty)$ be upper semi-continuous and $v\leq M< \infty$ on $X$. Then there exists a decreasing sequence of uniformly continuous functions $\{f_n\}$ on $X$ such that $f_n\leq M$ and for every $x\in X$, $f_n(x)$ decreases to $v(x)$.
\end{thm}

We will also need an analogue of the maximum modulus principle on unbounded domain for subharmonic functions. We briefly recall the definition of subharmonic functions. Let $D$ be an open subset of $\C$. A function $u : D \ra [-\infty, \infty)$ is called subharmonic if $u$ is upper semi-continuous and satisfies the local submean inequality, that is,  for any $w \in D$ there exists $\rho > 0$ such that for all $r\in (0, \rho)$ the following holds
\begin{equation*} 
u(w) \leq \frac{1}{2\pi} \int_0^{2\pi} u(w + r e^{it}) dt .
\end{equation*}
It is well known that if $f$ is a holomorphic function then $g(z)=\log (|f(z)|)$ is a subharmonic function (\cite{R}, P. 336). 
\begin{thm} [Theorem 7.15, P. 224, \cite{B}] \label{maxmod}
 Let $\Omega$ be a region (not necessarily bounded) and $u: \Omega \rightarrow \R$ be a subharmonic function which is bounded above. Let $A$ be a proper, countable subset of the boundary $\partial \Omega$ of $\Omega$ and $M$ a finite constant such that $\ds{\varlimsup_{z \to \xi} u(z)\leq M}$ for all $\xi \in \partial \Omega\smallsetminus  A$. Then $u \leq M$ throughout $\Omega$.
\end{thm}

\noindent \textit{Proof of Lemma \ref{lemma}:}
We consider the function 
\begin{equation*}
g(z)=\frac{1}{C}e^{iaz}f(z),\:\:\:\:\text{for all $z\in\C$,}
\end{equation*}
and observe that $g$ is an entire function. We want to apply Phragm\'{e}n-Lindel\"of theorem (Theorem 3.4, P. 124, \cite{SS}) to show that for all $z$ in the closed upper half plane $\overline{\mathbb H}=\{z\in \C: \Im z \geq 0\}$ 
\begin{equation}
|g(z)|\leq 1.\label{boundone}
\end{equation}
Let $Q_1 = \{z = x + iy \in \mathbb C: x > 0, y>0\}$. 
It follows from the estimate (\ref{expdecay}) that
\begin{equation*}
|g(iy)|=\frac{1}{C}e^{-ay}|f(iy)|\leq e^{-ay}e^{ay}=1,\:\:\:\:\text{for all $y>0$.}
\end{equation*}
It is also immediate from (\ref{decay}) that for all $x\in\R$
\begin{equation*}
|g(x)|=\frac{1}{C}|f(x)| \leq e^{-\theta(x)}\leq 1.
\end{equation*}
In particular, $g$ is bounded by $1$ on the positive real and positive imaginary axes. As $g$ satisfies the estimate (\ref{expdecay}) we can apply the Phragm\'{e}n-Lindel\"of theorem to the sector $Q_1$ to obtain (\ref{boundone}). A similar argument for the quadrant $Q_2 = \{z = x + iy \in \mathbb C: x < 0, y>0\}$ proves the estimate (\ref{boundone}) for all $z \in \overline{\mathbb H}$. Since $g$ is an entire function $\log |g|$ is subharmonic on $\C$ and 
\be 
\label{bound} \log |g(z)| \leq 0,\:\:\:\:\text{for all $z \in  \overline{\mathbb H}$.} 
\ee  
Now, we apply Theorem \ref{uppsemicont} for $X=\overline{\mathbb H}, v=\log |g|$, and $M=0$. Then there exists a decreasing sequence of uniformly continuous functions ${f_n}$ on $\overline{\mathbb H}$ such that $f_n \leq 0$ and $f_n(z)$ decreases to $v(z)$ for every $z\in \overline{\mathbb H }$. We define 
\begin{equation*}
v_n(z)= \max\{f_n(z), -n\},\:\:\:\:z\in \overline{\mathbb H},\:\:n \in \mathbb N.
\end{equation*} 
It is clear that $\{v_n\}$ is a decreasing sequence of continuous functions. As $f_n$ takes only negative values it follows that $v_n(z)\in [-n,0]$ for all $z\in\overline{\mathbb H}$. In particular, $v_n$ is bounded for each $n\in\mathbb N$. We now claim that $\{v_n\}$ converges pointwise to $v$ on $\overline{\mathbb H}$. We first assume that $v(z)=-\infty$. Then $\{f_n(z)\}$ and $\{-n\}$ both converge to $-\infty$ and hence so does $\{v_n(z)\}$. Now assume that $v(z)$ is finite. In this case $f_n(z)\in (v(z)-1, v(z)+1)$ for all large $n$ and hence $v_n(z)=f_n(z)$ for all large $n\in\mathbb N$. It follows that $\{v_n(z)\}$ converges to $v(z)$ for all $z\in\overline{\mathbb H}$.
Let $U_n$ be the Poisson integral of the restriction of the function $v_n$ on $\R$ given by
\be \label{U_n}
U_n(x+iy)=\frac{1}{\pi}\displaystyle\int_{\R}{\frac{yv_n(t)}{y^2+(x-t)^2}dt},\:\:\:\:x\in\R,\:y>0.
\ee
Since $v_n\in L^\infty(\R)$, the above integral exists and defines a harmonic function on the open upper half plane $\mathbb H=\{z\in \C: \Im z>0\}$. Moreover, since $v_n$ is continuous, we can extend $U_n$ to $\overline{\mathbb H}$ as a continuous function by letting $U_n(x) = v_n(x)$ for $x \in \R$ (Theorem 2.1(b), P. 47, \cite{SW}).
We now define 
\begin{equation}
V_n(z) = \log|g(z)|- U_n(z),\:\:\:\:z\in\mathbb H.\label{Vn}
\end{equation}
As $U_n$ is harmonic it follows that $V_n$ is subharmonic on $\mathbb H$. Since
\begin{equation*}
v_n(t) \geq -n,\:\:\:\:\text{for all $t\in\R$,}
\end{equation*}
it follows from the definition of $U_n$ given in (\ref{U_n}) that 
\begin{equation*}
U_n(z)\geq -n,\:\:\:\:z\in\mathbb H.
\end{equation*}
It now follows from (\ref{bound}) that 
\begin{equation*}
V_n(z) \leq n,\:\:\:\:\text{for all $z\in \mathbb H$.}
\end{equation*}
In particular, $V_n$ is bounded above for each $n\in\mathbb N$. Since 
\begin{equation*}
v(z) = \log|g(z)| \leq v_n(z),\:\:\:\:\text{for all $z \in \overline{\mathbb H}$,}
\end{equation*}
it follows that
\begin{equation*}
\lim_{y \to 0} V_n(x+iy)=\log |g(x)|-v_n(x) \leq 0,\:\:\:\:\text{for all $x\in\R$.}
\end{equation*}
We now apply Theorem \ref{maxmod} for $\Omega= \mathbb H$, $u= V_n$ and $A= \phi$, the empty set to conclude that
\begin{equation*}
V_n(z) \leq 0,\:\:\:\:\text{for all $z\in \mathbb H$.}
\end{equation*}
It follows from (\ref{Vn}) that 
\begin{equation*}
\log|g(x+iy)|\leq \frac{1}{\pi}\int_{\R}{\frac{yv_n(t)}{y^2+(x-t)^2}dt}, \:\:\:\:\text{for all $y>0$, $x\in\R$}.
\end{equation*} 
Since $\{v_n\}$ is a decreasing sequence, by using monotone convergence theorem and taking limit as $n\rightarrow \infty$ in the inequality above we get \bes
\log|g(x+iy)| \leq \frac{1}{\pi}\displaystyle\int_{\R}{\frac{y\log|g(t)|}{y^2+(x-t)^2}dt}.
\ees 
The estimate (\ref{decay}) now implies that
\begin{equation*}
\log|g(x+iy)| \leq -\frac{1}{\pi}\int_{\R}{\frac{y \theta(t)}{y^2+(x-t)^2}dt} \leq -C_{x,y}\int_{\R}{\frac{\theta(t)}{1+t^2}dt} = - \infty,
\end{equation*}
where $C_{x,y}$ is a positive constant which depends on $x$ and $y$.
So, for each $x\in \R$ and $y$ positive, it follows that $g(x+iy)=0$. As $f$ is an entire function it follows that $f(z)=0$ for all $z\in \C$.

\section{Euclidean Motion group}
Let $G=M(n)$ be the Euclidean motion group given by the semi-direct product of $\R^n$ with the special orthogonal group $K= SO(n)$. The group operation in $G$ is given by 
\begin{equation*}
 (x_1,k_1)(x_2,k_2) = (x_1 + k_1 \cdot x_2, k_1k_2),
\end{equation*} 
for $(x_1, k_1)$, $(x_2,k_2)$ in $G$ where $x_i \in \R^n$ and $k_i \in K$, $i=1,2$  and $k \cdot x$ denotes the natural action of $SO(n)$ on $\R^n$. If $dx$ denotes the Lebesgue measure on $\R^n$ and $dk$ the normalized Haar measure on $K$, then the Haar measure on $G$ is given by $dx~dk$. 

We shall now describe the unitary dual $\widehat G$, the equivalence classes of unitary, irreducible representations of $G$, as given in \cite{GK}. We consider a fixed non-zero $\xi \in \R^n$. If $U_\xi$ denotes the stabilizer of $\xi$ in $K$ under the natural action of $K$ on $\R^n$, then $U_\xi$ is conjugate to the subgroup 
\begin{equation*}\left\{\left(\begin{array}{cc}
      A & 0\\ 
      0 & 1
     \end{array}
\right):A\in SO(n-1)\right\},
\end{equation*}
 which we identify with $SO(n-1)$. Let $\lambda$ be an irreducible unitary representation of the compact group $U_\xi$ acting on $\C^{d_\lambda}$. Let
\beas
&&H(K,\lambda)\\&&=\{\psi : K \rightarrow \C^{d_\lambda}\:\:\text{measurable}\mid \psi(uk)=\lambda(u)\psi(k),\:u \in U_\xi,k\in K,  \int_K\|\psi(k)\|_{\C^{d_{\lambda}}}^2dk<\infty  \}.
\eeas
$H(K, \lambda)$ is a Hilbert space with respect to the inner product defined by the formula
\begin{equation*} 
\langle \psi_1, \psi_2 \rangle_{H(K,\lambda)} = d_\lambda\int_{K}{\langle\psi_1(k), \psi_2(k)\rangle_{\C^{d_\lambda}}dk},\:\:\:\:\psi_i\in H(K,\lambda), i=1,2.
\end{equation*} 
Now, we define a unitary representation $T_{\xi,\lambda}$ of $G$ on $H(K,\lambda)$ by 
\bes
\left(T_{\xi,\lambda}(x,k)\psi\right)(k_0)=e^{i\langle k_0^{-1}\cdot \xi,~ x\rangle}\psi(k_0k),
\ees
where $\psi\in H(K,\lambda), x\in \R^n$, and $k,k_0\in K$. It can be shown that (see \cite{GK}):
\begin{enumerate}
\item For $\xi\neq 0$ and any $\lambda\in \widehat U_\xi$, the representation $T_{\xi,\lambda}$ is irreducible.
\item Every infinite dimensional irreducible unitary representation of $G$ is equivalent to some  $T_{\xi,\lambda}$ with $\xi$ and $\lambda$ as above.
\item Given two non-zero vectors $\xi, \xi_1\in \R^n$ and representations $\lambda\in \widehat U_{\xi}$ and $\lambda_1\in \widehat U_{\xi_1}$, the representations $T_{\xi, \lambda}$ and $T_{\xi_1, \lambda_1}$ are equivalent if and only if $\xi$ and $\xi_1$ are in same $K$-orbit i.e. $\xi$ and $\xi_1$ have same Euclidean norm and the representations $\lambda$ and $\lambda_1$ are equivalent under the obvious identification of $U_\xi$ and $U_{\xi_1}$.
\end{enumerate}
It follows from the above equivalence of representations that instead of $T_{\xi,\lambda}$, $\xi\neq 0$, we can work with $T_{r,\lambda}$ for $r>0$ with $\|\xi\|=r$ and view it as $(r,0,\ldots ,0)\in\R^n$. We can also identify $U_r$ with $K_1 = SO(n-1)$. Apart from these infinite dimensional representations $T_{r,\lambda}$, the finite dimensional unitary representations of $K$ also yield finite dimensional unitary representations of $G$, but these do not enter into the Plancherel formula (see \cite{GK} for details). The Plancherel measure is supported on the subset of $\widehat G$ given by $\{T_{r,\lambda}: \lambda\in \widehat {K_1}, r\in \R^+\}$ and on 
each `piece' $\{T_{r,\lambda}: r\in \R^+\}$ with $\lambda\in \widehat{K_1}$ fixed, it is given by $C_n r^{n-1}dr$, where $C_n$ is a constant depending only on $n$.

Given a function $f \in L^1(G)$ and $\pi \in \widehat{G}$, the operator-valued group Fourier transform $\widehat{f}$ of $f$ at $\pi$ is given by the operator valued integral 
\begin{equation*}
\widehat{f}(\pi) = \int_{K} \int_{\R^n} f(x,k) \pi(x,k) dx dk.
\end{equation*}
It is known that for $f \in L^1 \cap L^2 (G)$ the group Fourier transform $\widehat{f}(\pi)$ is a Hilbert-Schmidt operator for almost all $\pi$ with respect to the Plancherel measure and we denote its Hilbert-Schmidt norm by $\|\widehat{f}(\pi)\|_{HS}$. We will now state and prove an analogue of Theorem \ref{paleywiener} for $M(n)$ in terms of the decay of the Hilbert-Schmidt norm of the group Fourier transform.

\begin{thm}\label{mn}
Let $f$ be a compactly supported continuous function on $M(n)$ satisfying the estimate  
\begin{equation}
 \label{mndecay} \|\widehat f(T_{r,\lambda})\|_{HS}\leq C_\lambda e^{-\theta(r)}, \:\:\:\: \txt{ for }r\in (0,\infty),
\end{equation}
where $\theta$ is a non-negative locally integrable function on $[0,\infty)$. If 
\begin{equation*}
I = \int_{0}^{\infty}{\frac{\theta(r)}{1+r^2}dr}=\infty,
\end{equation*} 
then $f= 0$ on $M(n)$.

Conversely, if $\theta : [0,\infty) \ra [0,\infty)$ is a non-decreasing function such that $I$ is finite, then there exists a non-zero $f \in C_c(M(n))$ satisfying (\ref{mndecay}). 
\end{thm}

\begin{proof}
Since $T_{-r,\lambda}$ and $T_{r,\lambda}$ are equivalent as representations of $M(n)$, we can extend $\theta$ as an even function on $\R$ to get that 
\bes \|\widehat f(T_{r,\lambda})\|_{HS}\leq C_\lambda e^{-\theta(r)}, \:\:\:\: \textmd{ for all } r\in \R.\\ \ees  For $r\in \R$ and $\lambda\in \widehat {K_1}$, let $\{e^\lambda_i : i\in \mathbb N\}$ be a basis of $H(K,\lambda)$
consisting of $K$-finite vectors. It suffices to show that for any fixed $i,j\in \mathbb N$, 
\bes \langle \widehat f(T_{r, \lambda})e^\lambda_i, e^\lambda_j\rangle_{H(K,\lambda)}=0\\ \ees
as a function of $r$ and $\lambda$. Fix $i_0, j_0 \in \mathbb N$ and consider for $r\in \R$ and $\lambda\in\widehat{K_1}$ the matrix entry of the group Fourier transform $\widehat{f}(T_{r,\lambda})$ given by
\begin{equation*}
\langle \widehat f(T_{r, \lambda})e^\lambda_{i_0}, e^\lambda_{j_0}\rangle_{H(K,\lambda)} =\int_{K}{\int_{\R^n}{f(x,k)\langle T_{r, \lambda}(x,k)e^\lambda_{i_0},e^\lambda_{j_0}\rangle_{H(K,\lambda)}}dx~dk}.
\end{equation*} 
Let us define
\begin{equation*}
\Phi^{i_0, j_0}_{r, \lambda}(x,k):= \langle T_{r, \lambda}(x,k)e^\lambda_{i_0},e^\lambda_{j_0}\rangle_{H(K,\lambda)},\:\:\:\:r\in \R,\lambda\in \widehat{K_1},
(x,k) \in M(n).
\end{equation*}
Using the description of the representation $T_{r,\lambda}$ we write
\begin{eqnarray}
 \Phi^{i_0, j_0}_{r, \lambda}(x,k) &=& d_\lambda\int_{K}{\langle (T_{r,\lambda}(x,k)e^\lambda_{i_0})(k_0), e^\lambda_{j_0}(k_0)\rangle_{\C^{d_\lambda}} dk_0}\nonumber\\
&=& d_\lambda\int_{K}{e^{i\langle k_0^{-1}\cdot r, ~x\rangle}\langle e^\lambda_{i_0}(k_0k), e^\lambda_{j_0}(k_0)\rangle_{\C^{d_\lambda}} dk_0}\nonumber\\
&=& d_\lambda\int_{K}{e^{i\langle r,~k_0\cdot x\rangle}\langle e^\lambda_{i_0}(k_0 k), e^\lambda_{j_0}(k_0)\rangle_{\C^{d_\lambda}} dk_0},\nonumber
\end{eqnarray}
where, as before, $r$ is identified with $(0, \cdots, 0, r) \in \R^n$. The integral on the right hand side makes sense even for $r\in \C$ and for fixed $(x,k)$, the function $r \mapsto \Phi^{i_0, j_0}_{r, \lambda}(x,k)$ extends to the whole of $\C$ as an entire function. For $r=t+is\in\C$ we easily obtain the estimate 
\begin{equation}
|\Phi^{i_0,j_0}_{r, \lambda}(x,k)| \leq C_\lambda \int_{K}{|e^{i\langle r,~ k_0 \cdot x\rangle}|dk_0} \leq C_\lambda e^{|r| \|x\|},\:\:\:\:x\in\R^n, k\in K.\label{matrixestimate} 
\end{equation}
For $r \in \C$ we define 
\begin{equation*}
g(r) := \langle \widehat f(T_{r,\lambda})e^\lambda_{i_0}, e^\lambda_{j_0}\rangle_{H(K,\lambda)} = \int_{K}\int_{\R^n} f(x,k) \Phi^{i_0,j_0}_{r, \lambda}(x,k) ~dx~dk .
\end{equation*} 
Since $f$ is compactly supported the function above is well defined. Moreover, it is an entire function on $\C$. It follows from (\ref{matrixestimate}) that there exists $A>0$ such that for all $r \in \C$
\begin{equation}
|g(r)| \leq C_\lambda \int_{K}{\int_{\R^n}{|f(x,k)| e^{|r| \|x\|}}dx~dk} \leq C_\lambda e^{A|r|}.\label{1stestimate}
\end{equation}
For $r\in \R$ the function $g$ satisfies the estimate 
\begin{equation}
|g(r)| = |\langle \widehat f(T_{r, \lambda})e^\lambda_{i_0}, e^\lambda_{j_0}\rangle_{H(K,\lambda)}| 
\leq \|\widehat f(T_{r, \lambda})\|_{HS} \leq C_\lambda e^{-\theta(r)}.\label{2ndestimate}
\end{equation}
Using the estimates (\ref{1stestimate}), (\ref{2ndestimate}) and applying Lemma \ref{lemma} to the function $g$ we get that $g=0$ and hence $f=0$ by the Plancherel theorem.

Conversely, let $\theta : [0,\infty) \ra [0,\infty)$ be a non-decreasing function such that $I < \infty$. By Theorem \ref{ccradon} (b), we get a radial function $g \in C_c^{\infty}(\R^n)$ such that 
\bes 
|\hat{g}(y)|\leq Ce^{-\theta(\|y\|)}, \:\:\:\: \txt{ for } y \in \R^n. 
\ees
We define
\begin{equation*}
f(x,k) = g(x),\:\:\:\:(x,k) \in M(n).
\end{equation*}
Thus $f$ is left and right invariant under the action of $SO(n)$. Since $(M(n),SO(n))$ is a Gelfand pair and $f$ is right $SO(n)$-invariant, it follows that for $r \in (0,\infty)$, $\widehat{f}(T_{r,\lambda})$ is zero unless $\lambda$ is the trivial representation of $SO(n-1)$. Moreover for the trivial representation $\lambda$ of $SO(n-1)$ it follows that $\langle \widehat{f}(T_{r,\lambda})\chi_n,\chi_m\rangle$ is zero unless $n=0$ where $\chi_k$'s are the spherical harmonics which form an orthonormal basis of $L^2(S^{n-1}) \cong H(K,\lambda)$. Further, since $f$ is left $SO(n)$-invariant, the only non-zero matrix entry of $\widehat{f}(T_{r,\lambda})$ is $\langle \widehat{f}(T_{r,\lambda})\chi_0,\chi_0\rangle = \hat{g}(r)$ where $\chi_0$ is the constant function $1$. It follows that $\|\widehat{f}(T_{r,\lambda})\|_{HS}^2 = |\hat{g}(r)|^2$ and so $f \in C_c(M(n))$ satisfies (\ref{mndecay}). 
\end{proof}
\begin{rem} Note that instead of the condition (\ref{mndecay}), if we assume that 
\begin{equation*}
|\langle \widehat f(T_{r, \lambda})e^\lambda_{i}, e^\lambda_{j}\rangle| \leq C_\lambda e^{-\theta(r)},
\end{equation*}
for every $i,j$ in the statement of Theorem \ref{mn} then the above proof goes through. This is in fact a weaker condition on the Fourier transform which gives a more general result.
\end{rem}

\section{Uniqueness of Solutions to The Schr\"odinger Equation}
It is now well known that the uniqueness of solutions to the initial value problem for the time-dependent Schr\"odinger Equation \begin{eqnarray}\label{schr}
\left\{\begin{array}{rcll} 
\displaystyle{\frac{\partial u}{\partial t}(x,t) ~ - ~ i\Delta u (x,t) } &=& ~ 0,
&\textmd{for } (x,t) \in \R^n \times [0, \infty), \\
u(x,0) &=& f(x), &\textmd{for } x \in \R^n,
\end{array}\right.
\end{eqnarray} 
is related to the uncertainty principles (see \cite{EKPV} and the references therein). Analogous results have also been obtained for certain non-commutative groups in \cite{Ch, PS, BTD, LM}. In this section we wish to deduce a similar uniqueness result in the context of $M(n)$. We first deduce one such result for $\R^n$.
\begin{thm}\label{uniqrn}
Let $u $ be a solution of the equation (\ref{schr}) and $\theta$ be a non-negative locally integrable function on $[0,\infty)$. Assume that $f\in C_c(\R^n)$ and for some positive $t_0$
\be \label{schrdecayrn} |u(x,t_0)| \leq Ce^{-\theta(\|x\|)}, \:\:\:\: \textmd{ for } x\in\R^n.\ee  If $\ds{\int_{0}^{\infty}{\frac{\theta(r)}{1+r^2}dr}=\infty,}$ then $u=0$. 
\end{thm}

\begin{proof}
We recall that the fundamental solution of (\ref{schr}) is given by 
\bes
u(x,t)=(e^{it\Delta}f)(x)=\frac{e^{i\frac{\|x\|^2}{4t}}}{(4\pi it)^{\frac{n}{2}}} \displaystyle\int_{\R^n}{e^{-i\frac{x \cdot y}{2t}}e^{i\frac{\|y\|^2}{4t}}f(y)dy}
\ees
for any $x \in \R^n$ and $t \in \R$. We define 
\bes 
g(x)=e^{\frac{i\|x\|^2}{4t_0}}f(x), \:\:\:\: \txt{ for } x \in \R^n.
\ees 
Then $g \in C_c(\R^n)$ and 
\bes 
|u(x,t_0)|=\frac{1}{{(4\pi|t_0|)}^{\frac{n}{2}}}\left|\hat g\left({\frac{x}{2t_0}}\right)\right|, \:\:\:\: \txt{ for } x \in \R^n.
\ees 
From (\ref{schrdecayrn}) and Theorem \ref{ccradon} we get $g$ is zero. It follows that $f$ is zero and hence so is $u$.
\end{proof}

In order to study solutions of the Schr\"odinger Equation on Euclidean Motion Group $M(n)$ we need to know the Laplacian on $M(n)$ explicitly. Following Hulanicki (see \cite{H}) and others, for a Lie group $G$ with polynomial growth, if we consider the elements of the Lie algebra $\mathfrak g$ of $G$ as left invariant differential operators on $C^{\infty}(G)$, then an analogue of the Laplacian $\Delta_G$ on $G$ can be defined as 
\bes
\Delta_G = X_1^2 + \cdots + X_k^2,
\ees 
where $\{X_1, \cdots, X_k\}$ is a basis of $\mathfrak g$. We recall that the Casimir operator $\Delta_K$ on a compact, connected Lie group $K$ is defined by 
\bes 
\Delta_K = T_1^2 + \cdots + T_N^2,
\ees 
where $T_1, T_2, \cdots, T_N$ is an orthonormal basis of the Lie algebra $\mathfrak k$ of $K$ with respect to a fixed Ad-$K$-invariant inner product on $\mathfrak k$. It is known that $\Delta_K$ coincides with the Laplace-Beltrami operator for the bi-invariant metric on $K$ determined by the fixed inner product on $\mathfrak k$ (see P. 107, \cite{Ha} and \cite{Ta}). The following lemma gives an explicit expression of the Laplacian $\Delta_G$ on $M(n)$ in terms of the standard Laplacian $\Delta_{\R^n}$ on $\R^n$ and the Laplace-Beltrami operator $\Delta_{SO(n)}$ on $SO(n)$:

\begin{lem}
$\Delta_G = \Delta_{\R^n} + \Delta_{SO(n)}$, \: for $G = M(n).$
\end{lem}

\begin{proof}
The group $M(n)$ may be identified with a matrix subgroup of $GL(n+1,\R)$, the group of invertible real matrices of order $n+1$, with the group operation in $M(n)$ corresponding to matrix multiplication via the map 
\bes (x, k) \mapsto \left (\begin{matrix} k & x \\ 0 & 1 \\ \end{matrix} \right)\ees 
for $x \in \R^{n}$ and $k \in SO(n)$. The Lie algebra $\mathfrak m(n)$ of $M(n)$ is given by the following Lie subalgebra of $\mathfrak{gl}(n+1,\R)$ 
\bes \left \{ X = \left ( \begin{matrix} T & E \\ 0 & 0 \\
\end{matrix} \right) : E \in \R^{n}, ~ T \in  \mathfrak{so}(n) \right \},\ees
where $\mathfrak{gl}(n+1,\R)$ is the Lie algebra of $GL(n+1,\R)$ consisting of real matrices of order $n+1$ with the standard Lie bracket $[A,B]=AB-BA$ for $A,B \in \mathfrak{gl}(n+1,\R)$ and $\mathfrak{so}(n) = \{ A \in \mathfrak{gl}(n,\R) : A+ A^t = 0\}$, the Lie algebra of $SO(n)$. It is easy to prove by induction that for any $m \in \mathbb N$,
\bes
 X^m = \left ( \begin{matrix} T^m & T^{m-1}E \\ 0 & 0 \\
\end{matrix} \right).
\ees
It follows that the image of $X$ under the matrix exponential map is given by \be \label{exp}  \exp X = I+X+\frac{X^2}{2!}+\frac{X^3}{3!}+\cdots = \left ( \begin{matrix} \exp T & (I+\frac{T}{2!}+\frac{T^2}{3!}+\cdots)E \\ 0 & 1 \\
\end{matrix} \right).\ee
Let $T_1, T_2, \cdots, T_N$ be an orthonormal basis of $\mathfrak{so}(n)$ (orthonormal with respect to the fixed Ad-$SO(n)$-invariant inner product on $\mathfrak{so}(n)$) and $E_1, E_2$, $\cdots, E_n$ be the standard basis of $\R^{n}$. Then a basis of $\mathfrak m (n)$ consists of the following elements 
\beas
X_i &=& \left ( \begin{matrix} T_i & 0 \\ 0 & 0 \\
\end{matrix} \right) \txt { for } 1 \leq i \leq N \\
X_{N+i }&=& \left ( \begin{matrix} 0 & E_i \\ 0 & 0 \\
\end{matrix} \right) \txt { for } 1 \leq i \leq n.
\eeas 
It follows from (\ref{exp}) that for $t \in \R$
\beas \exp tX_i &=& \left ( \begin{matrix} \exp{tT_i} & 0 \\ 0 & 1 \\ \end{matrix} \right) = (0, \exp tT_i) \in M(n) \txt{ for } 1 \leq i \leq N, \\ 
\exp tX_{N+i} &=& \left ( \begin{matrix} I & tE_i \\ 0 & 1 \\
\end{matrix} \right) = (tE_i,I) \in M(n) \txt{ for } 1 \leq i \leq n, \eeas where $I$ is the identity matrix of order $n$. It follows that the corresponding left invariant differential operators acting on $f \in C^{\infty}(M(n))$ are given by 
\beas
X_i f(x,k) &=& \left.\frac{d}{dt}\right|_{t=0} f((x,k)(0, \exp {t T_i}))\\ 
&=& \left.\frac{d}{dt}\right|_{t=0} f(x,k \exp{t T_i})\\ 
&=& T_i f (x,k)  \txt { for } 1 \leq i \leq N, \eeas
and \beas
X_{N+i} f(x,k)  &=& \left.\frac{d}{dt}\right|_{t=0} f((x,k)(tE_i,I))\\
&=& \left.\frac{d}{dt}\right|_{t=0} f(x+k\cdot tE_i ,k)\\
&=& \sum_{j=1}^n k_{ji} \frac{\partial f}{\partial x_j} (x,k)  \txt { for } 1 \leq i \leq n, 
\eeas
where $k_{ji}$ is the $ji$-th element of $k \in SO(n)$. It follows that 
\bes X_{N+i}^2 = \sum_{j,l=1}^n k_{ji} k_{li} \frac{\partial}{\partial x_j} \frac{\partial}{\partial x_l}. \ees
Summing over $i$ on both sides and using the orthogonality properties of the matrix entries of $k \in SO(n)$ we get that
\bes \sum_{i=1}^n X_{N+i}^2 = \sum_{j=1}^n \frac{\partial^2}{\partial x_j^2} = \Delta_{\R^n}.\ees  Hence \bes \Delta_G = \Delta_{\R^n} + \sum_{i=1}^N T_i^2 = \Delta_{\R^n} + \Delta_{SO(n)}.  \ees This completes the proof.
\end{proof}

The initial value problem for time-dependent Schr\"odinger equation on $M(n)$ is given by
\bea \label{schrG}
\left\{\begin{array}{rcll} 
\ds{\frac{\partial u}{\partial t}(g,t)} &=& i\Delta_G u(g,t),  \:\:\:\: &\txt{ for } g\in M(n), ~ t\in{\R}, \\
u(g,0) &=& f(g),  \:\:\:\: &\txt{ for } g \in M(n).
\end{array}\right.
\eea
It follows that the fundamental solution to (\ref{schrG}) is given by 
\be \label{fundsoln} u(g,t) = e^{it\Delta_G}f(g), \:\:\:\: \txt{ for } g \in M(n).\ee
For any $f \in L^2(M(n))$, expanding $f$ in the $SO(n)$-variable using the
Peter-Weyl theorem we obtain \be \label{pw} f(x,k) = \sum_{\pi \in
\widehat{SO(n)}} d_{\pi} \sum_{i,j = 1}^{d_\pi} f_{ij}^{\pi}(x)
\phi_{ij}^{\pi}(k),\ee where for each $\pi \in \widehat{SO(n)}$, $d_\pi$
is the degree of $\pi$, $\phi_{ij}^{\pi}$'s are the matrix
coefficients of $\pi$ and 
\be \label{f_ij} f_{ij}^{\pi}(x) = \int_{SO(n)} f(x,k) \overline{\phi_{ij}^{\pi}(k)} dk. \ee 
Here, the convergence is understood in the $L^2$-sense. Let $\lambda_{\pi}>0$ be such that $d\pi(\ds{\Delta_{SO(n)}}) = - \lambda_{\pi}I$ where $d\pi$ denotes the differential of the representation $\pi$ (see P. 107, \cite{Ha}). Then it is easy to see that for $k \in SO(n)$ 
\bes \Delta_{SO(n)} \phi_{ij}^{\pi}(k) = - \lambda_{\pi} \phi_{ij}^{\pi}(k). \ees 
This is because of the fact that for any $X \in \mathfrak m (n)$, 
\bes X \phi_{ij}^{\pi}(k) =  \left.\frac{d}{dt}\right|_{t=0} \langle \pi(k \exp tX) e_j, e_i \rangle = \langle \pi(k) d\pi(X)e_j, e_i \rangle. \ees 
It is known that the Laplacian $\Delta_G$ on $M(n)$ is self-adjoint and it acts on $f \in C_c^{\infty}(M(n))$ as
\beas
\Delta_G f (x,k) &=& \sum_{\pi \in \widehat{SO(n)}} d_{\pi} \sum_{i,j = 1}^{d_\pi} (\Delta_{\R^n} f_{ij}^{\pi})(x) \phi_{ij}^{\pi}(k) + f_{ij}^{\pi}(x) (\Delta_{SO(n)}\phi_{ij}^{\pi})(k)\\ 
&=& - \sum_{\pi \in \widehat{SO(n)}} d_{\pi} \sum_{i,j = 1}^{d_\pi} \left(\int_{\R^n} (\|\xi\|^2 +
\lambda_{\pi}) \hat{f_{ij}^{\pi}}(\xi) e^{i\xi \cdot y} d\xi \right) \phi_{ij}^{\pi}(k)
.\eeas
So we can use the spectral theorem to define the Schr\"odinger operator $e^{it\Delta_G}$ which is explicitly given by the spectral representation 
\be \label{e^itDelta} e^{it\Delta_G} f (x,k) = \sum_{\pi
\in \widehat{SO(n)}} d_{\pi} \sum_{i,j = 1}^{d_\pi}
\left(\int_{\R^n} e^{-it(\|\xi\|^2 +
\lambda_{\pi})} \hat{f_{ij}^{\pi}}(\xi) e
^{i\xi \cdot y} d\xi \right) \phi_{ij}^{\pi}(k). \ee We now present a uniqueness result for solutions to the  Schr\"odinger equation (\ref{schrG}).

\begin{thm}\label{uniqG}
Let $f \in C_c^{\infty}(M(n))$ and u be the solution to the system (\ref{schrG}) satisfying
\be \label{schrdecay} |u((x,k),t_0)| \leq Ce^{-\theta(\|x\|)}, \ee for some non-zero $t_0\in{\R}$ where $\theta$ is a non-negative locally integrable function on $[0,\infty)$. If $\int_{0}^{\infty}{\frac{\theta(r)}{1+r^2}dr}=\infty$, then $u=0$ on $M(n)\times \R$.
\end{thm}

\begin{proof}
Since $e^{it\Delta_G}$ is a unitary operator on $L^2(M(n))$, it follows that $(e^{it_0\Delta_G}f)_{ij}^{\pi}(x)$ is well defined for almost every $x \in \R^n$. From (\ref{fundsoln}), (\ref{f_ij}) and (\ref{schrdecay}),  we get that for almost every $x \in \R^n$, 
\be \label{solndecay} |(e^{it_0\Delta_G}f)_{ij}^{\pi}(x)| \leq \int_{SO(n)} |(e^{it_0\Delta_G}f)(x,k)| dk \leq C e^{-\theta(\|x\|)}.\ee
On the other hand, (\ref{e^itDelta}) can be written as 
\be \label{e^itDeltaG} e^{it\Delta_G} f (x,k) = \sum_{\pi
\in \widehat{SO(n)}} d_{\pi} \sum_{i,j = 1}^{d_\pi}  e^{-it\lambda_{\pi}} (e^{it\Delta_{\R^n}}f_{ij}^{\pi})(x) \phi_{ij}^{\pi}(k).\ee
It follows from (\ref{pw}) and (\ref{e^itDeltaG}) that 
\bes
(e^{it\Delta_G}f)_{ij}^{\pi}(x) = (e^{it\Delta_{\R^n}}f_{ij}^{\pi})(x) e^{-it\lambda_{\pi}},
\ees 
so from (\ref{solndecay}) we get 
\bes |(e^{it_0\Delta_{\R^n}}f_{ij}^{\pi})(x)| \leq  C e^{-\theta(\|x\|)}.\ees
Since $f_{ij}^{\pi} \in C_c(\R^n)$, we can apply Theorem \ref{uniqrn} to $f_{ij}^{\pi}$ to get that $f_{ij}^{\pi}$ is zero for each $i,j$ and $\pi$. Hence $u$ is zero.

\end{proof}

\textbf{Acknowledgement.} We would like to thank Swagato K. Ray for suggesting this problem and for the many useful discussions during the course of this work.

\end{document}